\documentclass[a4paper, 11pt, abstract=true]{scrartcl}

\usepackage{amsmath}
\usepackage{amsthm}
\usepackage{amsfonts}
\usepackage{amssymb}
\usepackage{bbm}
\usepackage{graphicx}
\usepackage{mathtools}

\usepackage{listings}
\newtheorem{thm}{Theorem}[section]
\newtheorem{lem}[thm]{Lemma}
\newtheorem{prop}[thm]{Proposition}
\newtheorem{cor}[thm]{Corollary}

\newcommand{\tr}{\operatorname{tr}}
\newcommand{\id}{\mathbbm{1}}

\newcommand{\diag}{\operatorname{diag}}

\newcommand{\spec}{\operatorname{spec}}
\newcommand{\dist}{\operatorname{dist}}

\newcommand{\R}{\mathbb{R}}
\newcommand{\C}{\mathbb{C}}

\usepackage{abstract}

\usepackage[backend=bibtex, style=alphabetic]{biblatex}
\bibliography{literatur}
\usepackage{url}
\usepackage{color}
\usepackage[plainpages=false,breaklinks=true,pdftitle={An operational measure for squeezing},pdfauthor={Martin Idel, Sebastian Soto Gaona, Michael Wolf},pdfsubject={Perturbation Theory for Williamson's normal form}]{hyperref}
\definecolor{linkblue}{rgb}{0.1,0.2,.7}
\definecolor{citegreen}{rgb}{0.1,0.7,.2}
\hypersetup{
    colorlinks,
    citecolor=citegreen,
    filecolor=linkblue,
    linkcolor=linkblue,
    urlcolor=linkblue
}
\DeclareFieldFormat{doi}{%
  \mkbibacro{DOI}\addcolon\space
    \ifhyperref
      {\href{http://dx.doi.org/#1}{\nolinkurl{#1}}}
      {\nolinkurl{#1}}}

\allowdisplaybreaks[1]

\title{\vspace{-15mm}%
	Perturbation Bounds for Williamson's Symplectic Normal Form}
\author{%
	\large
	\textsc{Martin Idel\footnote{martin.idel@tum.de}, Sebati\'an Soto Gaona, Michael M. Wolf} \\[2mm]
	\normalsize	Zentrum Mathematik, Technische Universit\"{a}t M\"{u}nchen, 85748 Garching \\
	\vspace{-5mm}
	}
\date{}

\allowdisplaybreaks

\begin{document}

\maketitle

\begin{abstract}
Given a real-valued positive semidefinite matrix, Williamson proved that it can be diagonalised using symplectic matrices. The corresponding diagonal values are known as the symplectic spectrum. This paper is concerned with the stability of Williamson's decomposition under perturbations. We provide norm bounds for the stability of the symplectic eigenvalues and prove that if $S$ diagonalises a given matrix $M$ to Williamson form, then $S$ is stable if the symplectic spectrum is nondegenerate and $S^TS$ is always stable. Finally, we sketch a few applications of the results in quantum information theory.
\end{abstract}
\section{Introduction}
It is well-known that the eigenvalues of a Hermitian matrix are stable under small perturbations. This lies at the heart of perturbation theory, which in turn is important in numerical analysis as well as most areas of theoretical physics. Many classic books have been fully devoted to (spectral) perturbation theory in finite \cite{wil65} or infinite \cite{kat95} dimensions and the results are well-known today.

Next to the general notion of eigenvalues of a linear operator, in symplectic linear algebra there exists the notion of \emph{symplectic eigenvalues}: Given a positive semidefinite matrix $M\in \R^{2n\times 2n}$, Williamson \cite{wil36} showed that there exists a symplectic basis such that $M$ is diagonal. The diagonal entries form the symplectic spectrum. The symplectic spectrum plays an important role in continuous variable quantum information, since for Gaussian states, the spectrum of the density matrix (which defines for instance the von Neumann entropy) can be obtained from the symplectic spectrum \cite{hol99}. The literature on symplectic eigenvalue perturbation is not as rich as for the usual eigenvalue problem. First results concerning perturbations for matrices in Williamson normal form can be found in \cite{ser05}. A more general approach was published in \cite{kon15}, and a bound similar to usual matrix perturbation bounds in the literature for matrix analysis has appeared recently in \cite{bha15}. The bounds in the present paper improve on the last result and also consider the Williamson analoga of eigenvectors and eigenspaces.

To summarise the results, let $M=S^TDS$ be the Williamson decomposition of a positive definite matrix. Since the symplectic spectrum is ultimately defined through the usual eigenvalue spectrum of a diagonalisable matrix, the immediate intuition is that the spectrum should be stable, while the diagonalising matrix $S$ will not be stable when symplectic eigenvalues are degenerate. Likewise, there is a chance that $S^TS$ is stable, because this would encode the information about ``symplectic eigenspaces'' and eigenspaces are generally stable \cite{bha96}. In accordance with this intuition, we find:
\begin{itemize}
	\item The symplectic spectrum is stable and we derive norm bounds for all unitarily invariant norms, improving the bounds in \cite{bha15}.
	\item The diagonalising matrix $S$ is stable as long as no eigenvalue crossings occur and we derive a norm bound depending on the smallest gap in the spectrum. We also give a counterexample for the stability of matrices with degenerate eigenvalues.
	\item $S^TS$ is stable and we derive norm bounds for the operator norm.
\end{itemize}

These results can be useful for proving continuity and approximation results at least in the context of continuous variable quantum information. We sketch a few applications in the last section. For the reader's convenience, we recall the most important theorems and a number of small lemmata with simple calculations in Appendix \ref{app:calc}.

\section{Notation and Williamson's normal form}
Throughout this paper, let $\sigma \in \R^{2n\times 2n}$ be the standard symplectic form defined as
\begin{align}
	\sigma=\begin{pmatrix}{} 0 & \id_n \\ -\id_n & 0 \end{pmatrix}.
\end{align}	
Furthermore, denote by $Sp(2n)$ the group of $2n\times 2n$ real symplectic matrices, by $O(n)\subseteq \R^{n\times n}$ the group of real orthogonal matrices, and by $U(n)\subseteq \C^{n\times n}$ the group of unitary matrices. Let us now define the symplectic spectrum through Williamson's theorem:
\begin{thm}[Williamson \cite{wil36}] \label{thm:williamson}
Let $M\in \R^{2n\times 2n}$ be a positive definite matrix. Then there exists a nonnegative diagonal matrix $D\in \R^{n\times n}$ and a symplectic matrix $S\in Sp(2n)$ such that
\begin{align}
	S^TMS=\diag(D,D).
\end{align}
We can assume without loss of generality that $D_{11}\geq D_{22}\geq \ldots \geq D_{nn}> 0$. The entries of $D$ are sometimes called the \emph{symplectic eigenvalues} of $M$ and they are the positive eigenvalues of $i\sigma M$.
\end{thm}
The theorem can be extended to the case of positive semidefinite matrices $M$. In this case, $S^TMS=\diag(D_1,D_2)$ where $D_1$ and $D_2$ contain zeros. 
\begin{proof}
One proof that covers also the case of semidefinite matrices can be found in \cite{deg06}. We sketch the proof of \cite{sim99} and \cite{sot15}:

Let $\diag(D,D)=:\tilde{D}$. Using that $D$ and $M$ are positive definite, consider an ansatz of the form $S=M^{-1/2}K\tilde{D}^{1/2}$, where $K\in O(2n)$. By construction $S^TMS=\tilde{D}$ and we only need to check that $K$ can be chosen such that $S$ is symplectic. This is equivalent to 
\begin{align*}
	K^T(M^{-1/2}\sigma M^{-1/2})K=\begin{pmatrix}{} 0 & D^{-1} \\ -D^{-1} & 0 \end{pmatrix}.
\end{align*}
Using that $(M^{-1/2}\sigma M^{-1/2})^T=-M^{-1/2}\sigma M^{-1/2}$, we know that we can indeed find an orthogonal $K$ achieving this construction. The idea is that $iM^{1/2}\sigma M^{1/2}$ is a Hermitian matrix and therefore diagonalisable by a unitary matrix $U\in U(2n)$. It is easy to see that the eigenvalues come in pairs $\pm \lambda_j$ with eigenvectors $x_j\pm iv_j$ for $j=1,\ldots,n$ and $x_j,y_j\in \R^{2n}$. One can then show that $K$ is given by $(x_1,\ldots, x_n,y_1,\ldots,y_n)$ using that $\sigma x_j=y_j$ and $\sigma y_j=-x_j$.
\end{proof}
The goal of the main part of this paper is to consider the stability of the symplectic eigenvalues, the diagonalising matrix $S$ and the matrix $S^TS$. 

\section{Stability of the symplectic eigenvalues}
Let us first consider the stability of $D$:
\begin{thm} \label{thm:stability}
Let $M,M^{\prime}\in \R^{2n\times 2n}$ be two positive definite matrices and $\tilde{D},\tilde{D}^{\prime}$ their Williamson diagonalisations as in Theorem \ref{thm:williamson}. Then
\begin{align}
	\|\tilde{D}-\tilde{D}^{\prime}\|\leq (\kappa(M)\kappa(M^{\prime}))^{1/2}\|M-M^{\prime}\|
\end{align}
for every unitarily invariant norm $\|\cdot \|$, where $\kappa(M)$ is the condition number of $M$.
\end{thm}
\begin{proof}
First note that by Williamson's theorem, $i\sigma M$ is diagonalisable. This can be seen via $S^{-1} (i\sigma M)S=i\sigma S^TMS=i\sigma \diag(D,D)$ and the fact that the latter has eigenvalues $\pm D_{jj}$ with eigenvectors $(0,\ldots,0,1,0,\ldots,0,\pm i,0,\ldots,0)^T$, where $1$ and $\pm i$ are at positions $j$ and $n+j$. Let $T$ be the matrix diagonalising $i\sigma \diag(D,D)$. Hence $i\sigma M$ is diagonalisable by $ST$ with real eigenvalues and the eigenvalues are given by $\pm D_{ii}$ for $i=1,\ldots n$. 

Using Lemma \ref{lem:bhatialemma} (\cite{bha96} (Theorem VIII.3.9)), we obtain directly:
\begin{align}
	\|\tilde{D}-\tilde{D}^{\prime}\|
		&\leq (\kappa(ST)\kappa(S^{\prime}T^{\prime}))^{1/2}\|i\sigma M-i\sigma M^{\prime}\| \\
		&= (\kappa(ST)\kappa(S^{\prime}T^{\prime}))^{1/2}\|M-M^{\prime}\|
\end{align}
for all unitarily invariant norms. We also used $\|i\sigma N\|=\|N\|$ for all Hermitian $N$ and all unitarily invariant norms, since $i\sigma$ is a unitary matrix. 

Since $i\sigma\diag(D,D)$ is Hermitian, its eigenvectors are orthogonal and we can choose $T\in U(2n)$. Therefore, since $\kappa(A)=\|A^{-1}\|_{\infty}\|A\|_{\infty}$ for all invertible matrices, $\kappa(ST)=\kappa(S)$ as the operator norm $\|\cdot\|_{\infty}$ is unitarily invariant. 

Let us now proceed as in \cite{sot15}: We can write $S=M^{-1/2}K\tilde{D}^{1/2}$ with an orthogonal matrix $K\in O(2n)$ using the proof of Williamson's theorem. Then
\begin{align*}
	\kappa(S)&=\|S\|_{\infty}\|S^{-1}\|_{\infty}= \|M^{-1/2}K\tilde{D}^{1/2}\|_{\infty}\|\tilde{D}^{-1/2}K^TM^{1/2}\|_{\infty} \\
		&\leq \|M^{-1/2}\|_{\infty} \|\tilde{D}^{1/2}\|_{\infty}\|\tilde{D}^{-1/2}\|_{\infty}\|M^{1/2}\|_{\infty} = \kappa(M^{1/2})\kappa(\tilde{D}^{1/2})
\end{align*}
Furthermore, 
\begin{align}
	\|\tilde{D}\|_{\infty}&\leq\max\{s(i\sigma M)\}=\|i\sigma M\|_{\infty} = \|M\|_{\infty} \label{eqn:DM}\\
	\|\tilde{D}^{-1}\|_{\infty}&\leq\max\{s(i\sigma M^{-1})\}=\|i\sigma M^{-1}\|_{\infty}= \|M^{-1}\|_{\infty} \label{eqn:DMinv}
\end{align}
where $s(A)$ denotes the vector of singular values of $A$. Using the $C^*$-property of the operator norm we obtain $\|(\tilde{D}^{1/2})^2\|_{\infty}=\|\tilde{D}^{1/2}\|^2_{\infty}$ and hence
\begin{align*}
	\kappa(S)\leq \kappa(M^{1/2})\kappa(\tilde{D}^{1/2})\leq \kappa(M^{1/2})\kappa(M^{1/2})=\kappa(M).
\end{align*}
Since the same is true for $M^{\prime}$ this completes the proof.
\end{proof}

In \cite{bha15}, the authors provide a different bound of this type, which for the operator norm reads
\begin{align}
	\|\tilde{D}-\tilde{D}^{\prime}\|_{\infty}\leq (\|M\|^{1/2}_{\infty}+\|M^{\prime}\|^{1/2}_{\infty})\|M-M^{\prime}\|_{\infty}^{1/2}.
\end{align}
Note that the scaling in $\|M-M^{\prime}\|_{\infty}$ is better in Theorem \ref{thm:stability} than in Bhatia and Jain's bound \cite{bha15}.

One natural question is, whether there is hope to improve a lot on this inequality. In particular, let us ask the question whether an inequality of the type 
\begin{align}
	\|\tilde{D}-\tilde{D}^{\prime}\|\leq c\|M-M^{\prime}\| \label{eqn:conj}
\end{align}
holds for some constant $c\in \R$ independent of $M, M^{\prime}$ and some unitarily invariant norm. The answer is ``no'':
\begin{prop}
Consider the following matrices:
\begin{align}
	M=\diag\left(x,1\right) \qquad E= \begin{pmatrix}{} 2 & -5 \\ -5 & -2\end{pmatrix}
\end{align}
Let $M_{\varepsilon}:=M+\varepsilon\cdot E$ for $\varepsilon>0$. Then, for all $0<\varepsilon<1/10$ and for all $c>0$ there exists an $x_0\geq 1$ such that for all $x\geq x_0$ we have
\begin{align}
	\|\tilde{D}-\tilde{D}_{\varepsilon}\|>c\|M-M_{\varepsilon}\|
\end{align}
for all unitarily invariant norms, thereby showing that $c$ must depend on $M$ and $M^{\prime}$ in equation (\ref{eqn:conj}).
\end{prop}
\begin{proof}
First note that $M_{\varepsilon}>0$ for $x\geq 1$ and $\varepsilon<1/10$, since trace and determinant are both positive. Note that $\|M-M_{\varepsilon}\|=\varepsilon \|E\|$. Now, since $\tr(E)=0$, the singular values of $E$ are both the same and given by $s(E)=\sqrt{29}$.

Since $\tilde{D}$ is two-dimensional and $M,M_{\varepsilon}$ are invertible, $\tilde{D}$ and $\tilde{D}_{\varepsilon}$ are multiples of the identity. Any unitarily invariant norm is a so called \emph{gauge function} of the singular values  (see \cite{bha96}, chapter IV). Since the matrices $\tilde{D}$ and $\tilde{D}_{\varepsilon}$ have only one singular value (excluding multiplicities), this implies that we can prove the statement for all unitarily invariant norm by proving it for the operator norm only.

Now we need to calculate the singular value of $\tilde{D}$ and $\tilde{D}_{\varepsilon}$, which is the positive eigenvalue of $i\sigma M$ and $i\sigma M_{\varepsilon}$ respectively. The characteristic polynomials are
\begin{align}
	\chi(i\sigma M)&=\lambda^2-x \\ 
	\chi(i\sigma M_{\varepsilon}) &=\lambda^2+5^2\varepsilon^2- \left(1-2\varepsilon\right)\left(x+2\varepsilon\right)
\end{align}
Note that for $x\geq 1$ and $\varepsilon$ small enough ($\varepsilon<1/10$ is sufficient), we have $\lambda_1(i\sigma M)=\sqrt{x}\geq \sqrt{x-2\varepsilon(x-1)-29\varepsilon^2}=\lambda_1(i\sigma M_{\varepsilon})$. Therefore, we have:
\begin{align}
	c\|M-M_{\varepsilon}\|_\infty-\|\tilde{D}-\tilde{D}_{\varepsilon}\|_\infty &< 0 \\
	\Leftrightarrow \quad \sqrt{29}c\varepsilon-|\sqrt{x}-\sqrt{x-2\varepsilon(x-1)-29\varepsilon^2}|&\leq 0 \\
	\Leftrightarrow \quad 2\sqrt{29x}c\varepsilon &\leq 29\varepsilon^2(1+c^2)+2\varepsilon(x-1) \label{eqn:counter1}
\end{align}
if we assume $x\geq 1$ and $\varepsilon<1/10$. For $c=1$, if $x\geq 33$, we have $\sqrt{29x}<(x-1)$ and therefore (independent of $\varepsilon$) $\|\tilde{D}-\tilde{D}_{\varepsilon}\|_{\infty}\geq \|M-M_{\varepsilon}\|_{\infty}$. Similarly, for any $c>0$ we can find $x_0\geq 0$, such that for all $x\geq x_0$ equation (\ref{eqn:counter1}) is satisfied, since $2(x-1)-\sqrt{29x}c\to +\infty$. 
\end{proof}
A further evaluation shows that if one sets $c=\kappa(M)^{1/2}\kappa(M_{\varepsilon})^{1/2}\approx x$ then $c\|M-M_{\varepsilon}\|_{\infty}$ scales as $x\varepsilon$ to lowest order in $\varepsilon$ (for $x\geq 1$), while $\|D-D_{\varepsilon}\|_{\infty}$ scales as $\sqrt{x}\varepsilon$. Therefore, the example above does not attain the bound. Whether the bound $c=\kappa(M)^{1/2}\kappa(M_{\varepsilon})^{1/2}$ is optimal can therefore not be determined by this counterexample. However, the scaling of $c$ in $x=\|M\|_{\infty}$ can only be improved by at most a square root.

\section{Stability of the diagonalising matrix \texorpdfstring{$S$}{S}}
Next, we will analyse stability of the matrix $S$. General wisdom from usual diagonalisation of Hermitian matrices tells us that this should hold at least when the eigenvalues are simple:
\begin{prop} \label{prop:stabilitys}
Let $M\in \R^{2n\times 2n}$ be a positive definite matrix such that all eigenvalues of $i\sigma M$ are nondegenerate and let $S\in Sp(2n)$ be the matrix diagonalising $M$ in Williamson's theorem. Let $E$ be a symmetric matrix with $\|E\|_{\infty}=1$, $\varepsilon>0$ such that $M_{\varepsilon}:=M+\varepsilon E$ is positive definite. Then we have:

For $\varepsilon>0$ small enough, the diagonalising matrix $S_{\varepsilon}\in Sp(2n)$ of $M_{\varepsilon}$ can be chosen in such a way that
\begin{align}
	\|S-S_{\varepsilon}\|_{\infty}<
		4\left(\sqrt{\kappa(M)}+\frac{\sqrt{n^{3}\|M\|_{\infty}/\|M^{-1}\|_{\infty}}}{2\delta}\right)
		\|M^{-1/2}\|_{\infty}\sqrt{\varepsilon}.
\end{align}
where $\delta:=\min_{i\neq j}|\lambda_i(i\sigma M)-\lambda_j(i\sigma M)|$ and $\kappa(M)$ is the condition number. 
\end{prop}
\begin{proof}
First observe that $\|S-S_{\varepsilon}\|_{\infty}=O(\sqrt\varepsilon)$ cannot be true for all choices of $S$ and $S_{\varepsilon}$ if those matrices are not unique, which occurs whenever there exists a matrix $O\in Sp(2n)\cap O(2n)$ that commutes with $M$ or $M^{\prime}$. 

We consider the construction of $S$ in the proof of Theorem \ref{thm:williamson} as $S=M^{-1/2}K\tilde{D}^{1/2}$ where $\tilde{D}=\diag(D,D)$. The stability of $S$ depends thus on the stability of $K$. Since the eigenvalues are simple, it is known that the eigenvectors are analytic functions in $\varepsilon$ (\cite{wil65}, chapter 2, section 5). Let $x_i(\varepsilon)$ denote the normalised eigenvectors of $i(M+\varepsilon E)^{-1/2}\sigma (M+\varepsilon E)^{-1/2}$. Then, using Lemma \ref{lem:wilkeigen}, there exists some constant $c_{\mathrm{vec}}$ such that for all $\varepsilon<c_{\mathrm{vec}}$ and all $i$ we have 
\begin{align}
	\|x_i-x_i(\varepsilon)\|_2\leq \frac{2n}{\min_{i\neq j}|\lambda_i(i\sigma M)-\lambda_j(i\sigma M)|}\varepsilon.
\end{align}
Note that $\lambda_j(iM^{1/2}\sigma M^{1/2})=\lambda_j(i\sigma M)$, $iM^{1/2}\sigma M^{1/2}$ is Hermitian and $\|E\|_{\infty}\leq 1$ fulfiling assumptions of Lemma \ref{lem:wilkeigen}.

We know that the parallelogram law holds for the vector norm $\|\cdot \|_2$:
\begin{align*}
	\|\Re(x_i)+i\Im(x_i)-\Re(x_i(\varepsilon))-i\Im(x_i(\varepsilon))\|_2^2 +\|\Re(x_i)-i\Im(x_i)-\Re(x_i(\varepsilon))+i\Im(x_i(\varepsilon))\|_2^2 \\
	=2\|\Re(x_i)-\Re(x_i(\varepsilon))\|_2^2+2\|\Im(x_i)-\Im(x_i(\varepsilon))\|_2^2
\end{align*}
Furthermore, we know that $K$ consists of the real and imaginary parts of eigenvectors of $M^{-1/2}\sigma M^{-1/2}$ and that when $x_i=\Re(x_i)+i\Im(x_i)$ is an eigenvector to the eigenvalue $\lambda_i$, then $x_i^{\prime}=\Re(x_i)-i\Im(x_i)$ is the eigenvector to the eigenvalue $-\lambda_i$. Thus we can find $K_{\varepsilon}$ such that:
\begin{align*}
	\|K-K_{\varepsilon}\|_{\infty}&\leq \left(\sum_{i=1}^{2n}\|K_{i}-(K_{\varepsilon})_{i}\|_2^2\right)^{1/2} \\
		&= \left( \sum_{i=1}^n \|\Re(x_i)-\Re(x_i(\varepsilon))\|_2^2+\sum_{i=1}^n \|\Im(x_i)-\Im(x_i(\varepsilon))\|_2^2\right)^{1/2} \\
		&= \left(\frac{1}{2} \sum_{i=1}^n \|x_i-x_i(\varepsilon)\|_2^2+\frac{1}{2} \sum_{i=1}^n\|x_i^{\prime}-x_i^{\prime}(\varepsilon)\|_2^2\right)^{1/2} \\
		&\stackrel{\mathclap{\mathrm{Lemma~ \ref{lem:wilkeigen}}}}{\leq}\quad~~ \left(\sum_{i=1}^n \frac{4n^2}{\min_{i\neq j}|\lambda_i(i\sigma M)-\lambda_j(i\sigma M)|^2}\varepsilon^2\right)^{1/2} \\
		&\leq \frac{2n^{3/2}}{\min_{i\neq j}|\lambda_i(i\sigma M)-\lambda_j(i\sigma M)|}\varepsilon
\end{align*}
where $K_i$ denotes the $i$-th column of $K$. Here, we used the fact that $\|A\|_{\infty}\leq \|A\|_F$ for the Frobenius norm, which is equivalent to the right hand side of the first inequality.

But then, using equation (\ref{eqn:DM}), Lemma \ref{lem:squareroot}, \ref{lem:inverse} and our stability result Theorem \ref{thm:stability} we obtain:
\begin{align}
	\|S-S_{\varepsilon}\|_{\infty}&= 	
			\|M^{-1/2}K\tilde{D}^{1/2}-M_{\varepsilon}^{-1/2}K_{\varepsilon}\tilde{D}_{\varepsilon}^{1/2}\|_{\infty}
				\nonumber \\
		&\leq \|M^{-1/2}-M_\varepsilon^{-1/2}\|_{\infty}\|\tilde{D}^{1/2}\|_{\infty}
			+\|M^{-1/2}_{\varepsilon}\|_\infty\|K-K_{\varepsilon}\|_{\infty}\|\tilde{D}^{1/2}\|_{\infty} \nonumber \\
		&~~~+\|M^{-1/2}_{\varepsilon}\|_{\infty} \|\tilde{D}^{1/2}-\tilde{D}^{1/2}_{\varepsilon}\|_{\infty} \nonumber \\
		&\leq \|M^{1/2}\|_{\infty}\|M^{-1/2}\|_{\infty}\|M^{-1/2}_\varepsilon\|_{\infty}\varepsilon^{1/2} 
			\label{eqn:terms1} \\
		&~~~+\|M_{\varepsilon}^{-1/2}\|_{\infty} (\kappa(M)\kappa(M_{\varepsilon}))^{1/4}\varepsilon^{1/2}
			\label{eqn:terms2} \\
		&~~~+\|M_{\varepsilon}^{-1/2}\|_{\infty}\|M^{1/2}\|_{\infty}
			\frac{2n^{3/2}}{\min_{i\neq j}|\lambda_i(i\sigma M)-\lambda_j(i\sigma M)|}\varepsilon \label{eqn:terms3}.
\end{align}
We can now use Lemma \ref{lem:inversenorm} and \ref{lem:kappaepsilon} to obtain for $\varepsilon<1/(2\|M^{-1}\|_{\infty})$ and $\varepsilon<\|M\|_{\infty}$:
\begin{align*}
	\|M_{\varepsilon}^{-1/2}\|_{\infty}\leq 2\|M^{-1/2}\|_{\infty}, \qquad
		\kappa(M_{\varepsilon})\leq 4\kappa(M).
\end{align*}
By assumption
\begin{align*}
	\min_{i\neq j}|\lambda_i(i\sigma M)-\lambda_j(i\sigma M)|\geq \delta.
\end{align*}
Hence we have for the summands in (\ref{eqn:terms1}) - (\ref{eqn:terms3})
\begin{align}
	\|M^{1/2}\|_{\infty}\|M^{-1/2}\|_{\infty}\|M^{-1/2}_\varepsilon\|_{\infty}\varepsilon^{1/2}
		&\leq \sqrt{2}\kappa(M)^{1/2}\|M^{-1/2}\|_{\infty}\varepsilon^{1/2} \\
	\|M_{\varepsilon}^{-1/2}\|_{\infty} (\kappa(M)\kappa(M_{\varepsilon}))^{1/4}\varepsilon^{1/2} 
		&\leq \sqrt{2}\|M^{-1/2}\|_{\infty}(4\kappa(M)^2)^{1/4}\varepsilon^{1/2} \\
	\|M_{\varepsilon}^{-1/2}\|_{\infty}\|M^{1/2}\|_{\infty}\frac{2n^{3/2}}{\min_{i\neq j}|\lambda_i-\lambda_j|}\varepsilon 
		&\leq (2\|M^{-1}\|_{\infty}\varepsilon)^{1/2}\|M^{1/2}\|_{\infty}\frac{2n^{3/2}}{\delta}\varepsilon^{1/2} \nonumber \\
		&\leq \|M^{1/2}\|_{\infty}\frac{2n^{3/2}}{\delta}\varepsilon^{1/2}		
\end{align}
where we used $\|M^{-1}\|_{\infty}\varepsilon<1/2$ by assumption on $\varepsilon$.

Put together, this implies that for all $0<\varepsilon<\min\{1/(2\|M^{-1}\|_{\infty}),\|M\|_{\infty}, c_{\mathrm{vec}}\}$,  we can find $S_{\varepsilon}$ diagonalising $M_{\varepsilon}$:
\begin{align}
	\|S-S_{\varepsilon}\|_{\infty}<
		2\left((1+1/\sqrt{2})\kappa(M)^{1/2}+\frac{n^{3/2}\|M\|^{1/2}_{\infty}}{\delta \|M^{-1}\|^{1/2}_{\infty}}\right)
		\|M^{-1/2}\|_{\infty}\varepsilon^{1/2}.
\end{align}
The constant is probably not optimal.
\end{proof}

However, this will not be true in general if we have eigenvalue crossings. To see this, consider the following counterexample: 
\begin{align}
	M=\begin{pmatrix}{} 1 & \varepsilon & 0 & 0 \\ \varepsilon & 1 & 0 & 0\\0 &0 & 1&\varepsilon \\ 0 & 0 & \varepsilon & 1 \end{pmatrix} \quad M^{\prime}=\diag(1+\varepsilon, 1-\varepsilon,1+\varepsilon, 1-\varepsilon)
\end{align}
By Williamson's theorem, for two matrices $S_1,S_2$ diagonalising $M$, we have $S_1^{-1}S_2\in Sp(2n)\cap O(2n)$ and $[S_1^{-1}S_2,M]=0$. Hence we need to consider the commutants of $M$ and $M^{\prime}$. For $\varepsilon>0$, an easy computation shows that $[M,O]=0$ or $[M^{\prime},O]=0$ if and only if 
\begin{align*}
	O=\begin{pmatrix}{} A & B \\ C & D \end{pmatrix},\quad A,B,C,D\in \R^{2\times 2},[A,E]=[B,E]=[C,E]=[D,E]=0,
\end{align*}
where $E\in \R^{2\times 2}$ denotes the upper left block in $M$ or $M^{\prime}$. This reduces the problem to finding the commutant of the upper left blocks in $M$. A simple computation shows that these are independent of $\varepsilon$. More precisely,
\begin{align*}
	[A,\diag(1+\varepsilon,1-\varepsilon)]=0 &\quad \Leftrightarrow \quad A=\diag(a,b),~a,b\in \R \\
	[A,\begin{pmatrix}{} 1 & \varepsilon \\ \varepsilon & 1 \end{pmatrix}]=0 & \quad \Leftrightarrow \quad A=\begin{pmatrix}{} a & b \\ b & a \end{pmatrix},~a,b\in \R.
\end{align*}
But then, the commutant of $M$ and $M^{\prime}$ are independent of $\varepsilon>0$ and so is the intersection of the commutant with $Sp(2n)\cap O(2n)$. Since this intersection is a closed set (commutants are closed), this implies that any matrix $S$ diagonalising $i\sigma M$ with $\|S\|_{\infty}=1$ and any matrix $S^{\prime}$ diagonalising $i\sigma M^{\prime}$ either fulfil $\|S-S^{\prime}\|_{\infty}>C$ for some fixed constant $C>0$ independent of $\varepsilon$, or there is a matrix $S$ diagonalising both $i\sigma M$ and $i\sigma M^{\prime}$. Since $[i\sigma M,i\sigma M^{\prime}]\neq 0$, the two matrices cannot be diagonalised simultaneously, whence $\|S-S^{\prime}\|_{\infty}$ cannot become arbitrarily close to zero. 

\section{Stability of the matrix \texorpdfstring{$S^{-T}S^{-1}$}{S(-T)S(-1)}}
We have seen that $S$ need not be stable when eigenvalue crossings occur, because eigenvectors need not be stable. However, it turns out that $S^{-T}S^{-1}$ is still stable because it contains only the (real parts of) projections onto the eigenspaces, which are stable according to general wisdom. In this section, $\|\cdot \|$ will always denote the norm $\|\cdot \|_{\infty}$ in order not to clutter the text with notation. 

\begin{thm} \label{thm:stabilitysmaller}
Let $M\in \R^{2n\times 2n}$ be a positive definite matrix with Williamson decomposition $M=S^{-T}\tilde{D}S^{-1}$.

Let $E$ be any symmetric matrix with $\|E\|=1$. Assume that $\varepsilon>0$ is small enough such that $M_\varepsilon:=M+\varepsilon E$ is still positive definite and let $M_{\varepsilon}=S^{-T}_{\varepsilon}\tilde{D}_{\varepsilon}S^{-1}_{\varepsilon}$ be its Williamson decomposition.

Then for any $\varepsilon$ such that $M_{\varepsilon}$ is positive definite and
\begin{align}
	0<\varepsilon < \min\left\{\frac{\|M\|}{(6\kappa(M))^{4/3}},\frac{1}{2\|M\|},\|M\|\right\} \label{eqn:epsilonassumption}
\end{align}
with the condition number $\kappa$, then
\begin{align}
	\|S^{-T}S^{-1}-S^{-T}_\varepsilon S^{-1}_\varepsilon \|\leq 9\pi n^3\kappa(M)^2 \|M^{-1}\|^{1/4} \varepsilon^{1/4}.
\end{align}
\end{thm}
The inequality can be improved by a more careful analysis of the prefactors.
\begin{proof}
From the proof of Theorem \ref{thm:williamson} we know $S=M^{-1/2}K\tilde{D}^{1/2}$ and therefore 
\begin{align*}
	S^{-T}S^{-1}=M^{1/2}K\tilde{D}^{-1}K^TM^{1/2}, 
\end{align*}
where $\tilde{D}=\diag(d_1,\ldots,d_n,d_1,\ldots,d_n)$ with $d_i>0$ and $K\in O(2n)$ is given by 
\begin{align}
	K=(v_1^{\Re},\ldots,v_n^{\Re},v_1^{\Im},\ldots,v_n^{\Im}).
\end{align}
Here, $v_i=v_i^{\Re}+iv_i^{\Im}$ are the eigenvectors of $iM^{1/2}\sigma M^{1/2}$ corresponding to $d_i$. We have
\begin{align}
	\|S^{-T}S^{-1}-&S^{-T}_{\varepsilon}S^{-1}_{\varepsilon}\| \nonumber\\
		&=\|M^{1/2}K\tilde{D}^{-1}K^TM^{1/2}-M_\varepsilon^{1/2}K_\varepsilon\tilde{D}_\varepsilon^{-1}K_\varepsilon^T M_\varepsilon^{1/2}\| \nonumber\\
	&\leq \|M^{1/2}K\tilde{D}^{-1}K^TM^{1/2}-M_{\varepsilon}^{1/2}K\tilde{D}^{-1}K^TM^{1/2}\| \nonumber\\
		&~~~+\|M_{\varepsilon}^{1/2}K\tilde{D}^{-1}K^TM^{1/2}- M_{\varepsilon}^{1/2}K_{\varepsilon}\tilde{D}_{\varepsilon}^{-1}K_{\varepsilon}^TM^{1/2}\| \nonumber\\
		&~~~+\|M_{\varepsilon}^{1/2}K_{\varepsilon}\tilde{D}_{\varepsilon}^{-1}K_{\varepsilon}^TM^{1/2} -M_{\varepsilon}^{1/2}K_{\varepsilon}\tilde{D}_{\varepsilon}^{-1}K_{\varepsilon}^TM_{\varepsilon}^{1/2}\| \nonumber\\
		&\leq \|M^{1/2}-M_{\varepsilon}^{1/2}\|\|M^{1/2}\|\|\tilde{D}^{-1}\|
		+\|M^{1/2}-M_{\varepsilon}^{1/2}\|\|M_{\varepsilon}^{1/2}\|\|\tilde{D}_{\varepsilon}^{-1}\| \nonumber\\
		&~~~+\|M_{\varepsilon}^{1/2}\|\|M^{1/2}\|\|K\tilde{D}^{-1}K^T-K_{\varepsilon}\tilde{D}_{\varepsilon}^{-1}K_{\varepsilon}^T\| \nonumber \\
		&\leq (\|M^{1/2}\|\|\tilde{D}^{-1}\|+\|M_{\varepsilon}^{1/2}\|\|\tilde{D}_{\varepsilon}^{-1}\|)\|M^{1/2}-M_{\varepsilon}^{1/2}\| \label{eqn:term1} \\
		&~~~+\|M_{\varepsilon}^{1/2}\|\|M^{1/2}\|\|\tilde{D}^{-1}-\tilde{D}_{\varepsilon}^{-1}\| \label{eqn:term2}\\
		&~~~+\|M_{\varepsilon}^{1/2}\|\|M^{1/2}\|\|K\tilde{D}^{-1}K^T-K_{\varepsilon}\tilde{D}^{-1}K_{\varepsilon}^T\| \label{eqn:term3}
\end{align}
We deal with each term separately, where the hardest term is the last. 

Term (\ref{eqn:term1}): Using (\ref{eqn:DMinv}) we have $\|\tilde{D}^{-1}\|\leq \|M^{-1}\|$. For $\|M^{-1}\|\varepsilon<1/2$ and $\varepsilon<\|M\|$, Lemma \ref{lem:squareroot} and \ref{lem:inversenorm} imply
\begin{align}
	(\|M^{1/2}\|\|\tilde{D}^{-1}\|+&\|M_{\varepsilon}^{1/2}\| \|\tilde{D}_{\varepsilon}^{-1}\|)\|M^{1/2}-M_{\varepsilon}^{1/2}\| \nonumber \\
	&\leq (\|M^{1/2}\|\|M^{-1}\|+4\|M^{1/2}\|\|\tilde{M}^{-1}\|)\varepsilon^{1/2} \label{eqn:term1f}\\
	&\leq 5\kappa(M)^{1/2} \|M^{-1}\|^{1/2}\varepsilon^{1/2}. \nonumber
\end{align}

Term (\ref{eqn:term2}): Since $\tilde{S}^{-1}$ diagonalises $M^{-1}\geq 0$, Theorem \ref{thm:stability} implies 
\begin{align*}
	\|\tilde{D}^{-1}-\tilde{D}_{\varepsilon}^{-1}\| &\leq (\kappa(M)\kappa(M_{\varepsilon}))^{1/2}\|M^{-1}-M_{\varepsilon}^{-1}\| \\ 
		&\stackrel{\mathclap{\mathrm{Lemma~}\ref{lem:inverse},\ref{lem:kappaepsilon}}}{\leq}\qquad 4\kappa(M) \|M^{-1}\|^2\varepsilon
\end{align*}
for $\|M^{-1}\|\varepsilon<1/2$ and with $\|M_{\varepsilon}\|\leq \|M\|+\varepsilon\leq 2\|M\|$. Plugging this into (\ref{eqn:term2}) and using $\varepsilon<\|M\|$ we obtain
\begin{align}
	\|M_{\varepsilon}^{1/2}\|\|M^{1/2}\|\|\tilde{D}^{-1}-\tilde{D}_{\varepsilon}^{-1}\| 
		&\leq 4\kappa(M)^{3/2} \|M^{-1}\|\varepsilon \label{eqn:term2f}
\end{align}

Term (\ref{eqn:term3}): The interesting part is $\|K\tilde{D}^{-1}K^T-K_{\varepsilon}\tilde{D}^{-1}K_{\varepsilon}^T\|$. We start by observing:
\begin{align}
	K\tilde{D}^{-1}K^T&=\sum_{i=1}^n d_i^{-1} v_i^{\Re}v_i^{\Re\,T}+\sum_{i=1}^n d_i^{-1} v_i^{\Im}v_i^{\Im\,T} 
		=\sum_{i=1}^n d_i^{-1}\qquad~~~ \sum_{\mathclap{j\in \{k|d_k=d_i,k=1,\ldots,n\}}} (v_j^{\Re}v_j^{\Re\,T}+v_j^{\Im}v_j^{\Im\,T}). \label{eqn:proof2}
\end{align}
Furthermore, 
\begin{align}
	\sum_{\mathclap{j\in \{k|d_k=d_i,k=1,\ldots,n\}}} (v_j^{\Re}v_j^{\Re\,T}+v_j^{\Im}v_j^{\Im\,T}) 
		~~~= \Re \sum_{\mathclap{j\in \{k|d_k=d_i,k=1,\ldots,n\}}} (v_j^{\Re}+iv_j^{\Im})(v_j^{\Re}+iv_j^{\Im})^{*}= \Re(P_M(d_i)) \label{eqn:proof1}
\end{align}
where $\Re$ denotes the real part of the expression and $P_M(d_i)$ denotes the spectral projection onto the eigenvalue subspace of the eigenvalue $d_i$ of $iM^{1/2}\sigma M^{1/2}$. We wish to apply general knowledge about the stability of eigenspaces. For convenience, the relevant theorem (\cite{bha96} Theorem VII.3.2) is stated in Lemma \ref{lem:stablesubspace}. 

In order to apply it, we need to consider the spectrum of $iM^{1/2}_{\varepsilon}\sigma M^{1/2}_{\varepsilon}$: By construction, $\spec(iM^{1/2}\sigma M^{1/2})=\{\pm d_1,\ldots,\pm d_n\}$. Denote the positive eigenvalues as $\spec_+$, then we can write $\spec_+(iM^{1/2}\sigma M^{1/2})=\bigcup_{j=1}^k S_j$ where all $S_j$ contain $d_i$ with multiplicities, fulfil $\operatorname{dist}(S_j,S_k):=\min\{|d-e|~|~d\in S_j,e\in S_k\}>\|M\|^{3/4}\varepsilon^{1/4}$ and $k$ is maximal. 

The stability of the symplectic spectrum implies:
\begin{align}
	\|\tilde{D}-\tilde{D}_{\varepsilon}\|\qquad &\stackrel{\mathclap{\mathrm{Theorem~\ref{thm:stability}}}}{<} \qquad (\kappa(M)\kappa(M_{\varepsilon}))^{1/2}\varepsilon \nonumber
	\\&\stackrel{\mathclap{\mathrm{Lemma~\ref{lem:kappaepsilon}}}}{\leq}\quad 2\kappa(M)\varepsilon 
		= \frac{6\kappa(M)}{\|M\|^{3/4}}\varepsilon^{3/4} \frac{\|M\|^{3/4}\varepsilon^{1/4}}{3} \nonumber \\
	&\stackrel{\mathclap{\mathrm{Assumption~(\ref{eqn:epsilonassumption})}}}{<}\qquad~\|M\|^{3/4}\varepsilon^{1/4}/3 \label{eqn:proof4}
\end{align}
hence if we set $d_i:=\tilde{D}_{ii}$ and $e_i:=(\tilde{D}_{\varepsilon})_{ii}$, then we have
\begin{align}
	|d_i-e_i|<\|M\|^{3/4}\varepsilon^{1/4}/3\quad \forall i=1,\ldots,n. \label{eqn:newmap}
\end{align}
We can now define the multisets $R_j:=\{e_i|d_i\in S_j\}$ for every $S_j$ and make the following observations:
\begin{enumerate}
	\item The diameter of $S_j$ does not exceed $\|M\|^{3/4}\varepsilon^{1/4}|S_j|$. \label{enu:prop1}
	\item $|R_j|=|S_j|$ for every $j=1,\ldots,k$. \label{enu:prop2}
	\item $\dist(R_i,R_j)>1/3\|M\|^{3/4}\varepsilon^{1/4}$ for $i\neq j$. \label{enu:prop3}
\end{enumerate}
Observation \ref{enu:prop1} follows from the maximal number of $S_j$: If the diameter was larger, by the pidgeon-hole principle we could divide $S_j$ into two sets with distance larger than $\|M\|^{3/4}\varepsilon^{1/4}$. 

Observation \ref{enu:prop2} follows, since any $e_i\in R_j$ is at most $1/3\|M\|^{3/4}\varepsilon^{1/4}$ away from some $d_i\in S_j$ and since the distance of $S_j$ and $S_k$ is at least $\|M\|^{3/4}\varepsilon^{1/4}$, $|e_i-d_k|>2/3\|M\|^{3/4}$ for any $d_k\in S_k$ with $k\neq j$. Incidentally, this also proves Observation \ref{enu:prop3}.

Now let $e_i$ be as defined with equation (\ref{eqn:newmap}). Using equation (\ref{eqn:proof1}), we see
\begin{align*}
	\|K\tilde{D}^{-1}K^T-K_{\varepsilon}\tilde{D}^{-1}K_{\varepsilon}^T\|=\left\| \sum_{j=1}^k \sum_{d_i\in S_j}d_i^{-1} (\Re(P_M(d_i))-\Re(P_{M_{\varepsilon}}(e_i)))\right\|
\end{align*}
For every set $S_j$, pick a value $d_{S_j}\in S_j$ and we have:
\begin{align}
	\|K\tilde{D}^{-1}K^T-&K_{\varepsilon}\tilde{D}^{-1}K_{\varepsilon}^T\|=\left\| \sum_{j=1}^k d_{S_j}^{-1} (\Re(P_M(S_j))-\Re(P_{M_{\varepsilon}}(R_j)))\right. \nonumber \\ 
		&~~~+\left.\sum_{j=1}^k \sum_{d_i\in S_k} (d_i^{-1}-d_{S_j}^{-1})\sum_{d_i\in S_j}(\Re(P_M(d_i))-\Re(P_{M_{\varepsilon}}(e_i)))\right\| \nonumber \\
		&\leq \sum_{j=1}^k d_{S_j}^{-1} \|\Re(P_M(S_j))-\Re(P_{M_{\varepsilon}}(R_j))\| \label{eqn:stepinter2} \\
		&~~~+\sum_{j=1}^k \sum_{d_i\in S_j} \frac{|d_i-d_{S_j}|}{d_id_{S_j}}\|\Re(P_M(d_i))-\Re(P_{M_{\varepsilon}}(e_i))\| \label{eqn:stepinter}
\end{align}
Recall that the real part of an operator $T$ is defined as $\Re(T):=(T+T^*)/2$. Since it is clearly linear for all matrices $T$, using that $\|T+T^*\|\leq 2\|T\|$ and the fact that every unitarily invariant norm fulfils $\|T^*\|=\|T\|$ implies 
\begin{align*}
	\|\Re(P_M(S_i))-\Re(P_{M_{\varepsilon}}(R_j))\|\leq \|P_M(S_i)-P_{M_{\varepsilon}}(R_j)\| \quad \forall i,j.
\end{align*}

Now we can apply Lemma \ref{lem:stablesubspace} to the term (\ref{eqn:stepinter2}): Let $P^c_M(S_j)$ and $P^c_{M_{\varepsilon}}(R_i)$ be complementary orthogonal projections such that in particular $P_M(S_j)+P_M^c(S_j)=\id$. Then we have for every $j=1,\ldots,k$:
\begin{align*}
	\|P_M(S_j)-P_{M_{\varepsilon}}(R_j)\|&=\|P_M(S_j)(P_{M_{\varepsilon}}(R_j)+P_{M_{\varepsilon}}^c(R_j))-(P_M(S_j)+P_M^c(S_j))P_{M_{\varepsilon}}(R_j)\| \\
	&=\|P_M(S_j)P_{M_{\varepsilon}}^c(R_j)-P_M^c(S_j)P_{M_{\varepsilon}}(R_j)\| \\
	&\leq \frac{3\pi}{2\varepsilon^{1/4}\|M\|^{3/4}} \|iM^{1/2}\sigma M^{1/2}-iM_{\varepsilon}^{1/2}\sigma M_{\varepsilon}^{1/2}\| \\
	&\leq \frac{3\pi}{2\varepsilon^{1/4}\|M\|^{3/4}}(\|M^{1/2}\|+\|M_{\varepsilon}^{1/2}\|)\|M^{1/2}-M_{\varepsilon}^{1/2}\|.
\end{align*}
Here, we used Observation \ref{enu:prop3} of the decomposition, which gives a lower bound on $\dist(S_j,R_i)$ for $i\neq j$. 

For the term (\ref{eqn:stepinter}) we use that the norm of the difference of two projections never exceeds one:
\begin{align}
	\sum_{j=1}^k \sum_{d_i\in S_j} \frac{|d_i-d_{S_j}|}{d_id_{S_j}}\|\Re(P_M(d_i))-\Re(P_{M_{\varepsilon}}(e_i))\| \leq \sum_{j=1}^k \sum_{d_i\in S_j} \frac{|d_i-d_{S_j}|}{d_id_{S_j}}. 
\end{align}
We can now use the Observation \ref{enu:prop1} and the fact that for any $S_j$, $|S_j|\leq n$. This gives an upper bound to $|d_i-d_{S_j}|$. Furthermore, $|d_id_{S_j}|\geq|d_{\mathrm{min}}|^2=1/\|\tilde{D}^{-1}\|^2$ and hence,
\begin{align}
 	\sum_{j=1}^k \sum_{d_i\in S_j} \frac{|d_i-d_{S_j}|}{d_id_{S_j}} \leq \sum_{j=1}^k n^2\|\tilde{D}^{-1}\|^2\|M\|^{3/4}\varepsilon^{1/4} \leq n^3\|\tilde{D}^{-1}\|^2\|M\|^{3/4}\varepsilon^{1/4}.	
\end{align}
In total, we obtain
\begin{align}
	\|K\tilde{D}^{-1}K^T-K_{\varepsilon}\tilde{D}^{-1}K_{\varepsilon}^T\|
		&\leq \sum_{j=1}^k d_i^{-1} \frac{3\pi}{2\varepsilon^{1/4}\|M\|^{3/4}}(\|M^{1/2}\|+\|M_{\varepsilon}^{1/2}\|)\|M^{1/2}-M_{\varepsilon}^{1/2}\| \nonumber \\ 
		&+n^3\|\tilde{D}^{-1}\|^{2}\|M\|^{3/4}\varepsilon^{1/4}. \label{eqn:proof6}
\end{align}
Now, we know that $\sum_{i=1}^k d_i^{-1}\leq\|\tilde{D}^{-1}\|_1\leq n\|M^{-1}\|$ by equation (\ref{eqn:DMinv}) and the fact that $\|\id\|_1=n$. Using Lemmata \ref{lem:squareroot}, \ref{lem:inversenorm} and \ref{lem:kappaepsilon}, we can now fully evaluate the term (\ref{eqn:term3}):
\begin{align}
	\|M_{\varepsilon}^{1/2}\|\|M^{1/2}\|&\|K\tilde{D}^{-1}K^T-K_{\varepsilon}\tilde{D}^{-1}K_{\varepsilon}^T\| \nonumber \\
	&\leq 2\|M\|\left(n\|M^{-1}\| \frac{9\pi}{2\varepsilon^{1/4}\|M\|^{3/4}}\cdot \|M\|^{1/2}\varepsilon^{1/2}+n^3\|M^{-1}\|^2\|M\|^{3/4}\varepsilon^{1/4}\right) \nonumber \\
	&\leq 9 \pi n \kappa(M)^{3/4} \|M^{-1}\|^{1/4}\varepsilon^{1/4}+n^3\kappa(M)^{7/4}\|M^{-1}\|^{1/4}\varepsilon^{1/4} \label{eqn:term3f}.
\end{align}

Finally, we can put everything together by substitution (\ref{eqn:term1f}), (\ref{eqn:term2f}) and (\ref{eqn:term3f}) into (\ref{eqn:term1})-(\ref{eqn:term3}). Using $\|M^{-1}\|\varepsilon<\|M^{-1}\|^{1/4}\varepsilon^{1/4}<1/2$, we obtain:
\begin{align}
	\|S^{-T}S^{-1}-&S^{-T}_{\varepsilon}S^{-1}_{\varepsilon}\|\leq 9 \pi n^3\kappa(M)^2 \|M^{-1}\|^{1/4}\varepsilon^{1/4}.
\end{align}
The constant is not optimal.
\end{proof}

\section{Applications}
Let us now sketch a few applications of the theorems to quantum information theory (an overview can be found in \cite{ade14}). The basic object in quantum mechanics are the quantum state of a system. In the case of systems consisting of $n$ bosonic modes (such systems are considered especially in quantum optics), an important set of states are the so called \emph{Gaussian states}. They can be characterised by their first and second moments, which correspond to a vector $d\in \R^{2n}$ and a covariance matrix $\gamma\in \R^{2n\times 2n}$. Necessary and sufficient conditions for $\gamma$ to be the covariance matrix of a quantum state are given as $\gamma\geq i\sigma$ by Heisenberg's inequality. \emph{Pure states} then correspond to symplectic positive definite matrices. Given two systems of $n$-modes and a state on two systems given by $\gamma_{AB}\in \R^{4n\times 4n}$, one can consider the \emph{reduced state} of the quantum system, which is given by the upper left $2n\times 2n$-submatrix of $\gamma_{AB}$.

An important quantity in quantum information is the entropy of entanglement. As proven in \cite{hol99}, the \emph{entanglement entropy} for Gaussian states is a continuous function of the symplectic spectrum of the reduced state of a system. Given a Gaussian quantum state with covariance matrix $\gamma$, it is given by
\begin{align}
	H(\gamma)=\sum_{k=1}^n \left(g\left(\frac{d_k+1}{2}\right)-g\left(\frac{d_k-1}{2}\right)\right)
\end{align}
where $g(x)=x\log(x)$ and the $d_k$ are the symplectic eigenvalues.

An easy corollary of Theorem \ref{thm:stability} is the following norm bound on the entropy difference:
\begin{cor} \label{cor:entropy}
For Gaussian states characterised by $(\gamma_{AB},d)$, the entropy of entanglement is continuous in $\gamma_{AB}$. Furthermore, for two states $\gamma$ and $\tilde{\gamma}$ in the interior of the set of covariance matrices, the entropy difference is bounded by
\begin{align*}
	|H(\gamma)-H(\tilde{\gamma})|\leq (\kappa(\gamma)\kappa(\tilde{\gamma})^{1/2})(1+\log(\max(\|\gamma\|_{\infty},(\|\gamma^{-1}\|_{\infty}^{-1}-1)/2)))\|\gamma-\tilde{\gamma}\|_1
\end{align*}
\end{cor}
\begin{proof}
The entropy is continuous, since $g$ is continuous and the symplectic eigenvalues are continuous.

Let $d_k$ be the entries of $D$, the Williamson diagonalisation of $\gamma$ (likewise $\tilde{d}_k$), which implies that $\frac{d_k+1}{2}\geq 1$ always and $\frac{d_k-1}{2}\geq 0$. For $x>0$ we have
\begin{align*}
	|x\log(x) - y\log(y)|=|x\log(x)-y\log(x)+y\log(x)-y\log(y)| \\
	=|(x-y)\log(x)+y\log(1+(x-y)/y)|\leq |\log(x)||x-y|+|x-y|
\end{align*}
For $x=0$, the upper bound is clearly also true, since $x\log(x)=0$. 
Using that $\log((d_k+1)/2)\leq \log(d_k)$, we obtain
\begin{align*}
	\left|\left(g\left(\frac{d_k+1}{2}\right)-g\left(\frac{d_k-1}{2}\right)\right)-\left(g\left(\frac{\tilde{d}_k-1}{2}\right)-g\left(\frac{\tilde{d}_k-1}{2}\right)\right)\right| \\
	\leq (1+\log(d_k))|d_k-\tilde{d}_k|+(1+\log((d_k-1)/2))|d_k-\tilde{d}_k|
\end{align*}
Taking the sum and noting that $d_k\leq \|D\|_{\infty}$ by assumption and $\min_{k} d_k\leq \|D^{-1}\|^{-1}_{\infty}$, we have
\begin{align*}
	|H(\gamma)-H(\tilde{\gamma})|&
		\leq (1+\log(\|D\|_{\infty}))\|D-\tilde{D}\|_{1}+(1+\log((1/\|D^{-1}\|_{\infty}-1)/2))\|D-\tilde{D}\|_{1}.
\end{align*}
The rest then follows by using Theorem \ref{thm:stability}.
\end{proof}
Note that the bound becomes arbitrarily bad if $D$ has eigenvalues close to one. This is to be expected, since the function $x\log(x)$ is not uniformly continuous at $0$ and hence cannot be norm bounded with a constant independent of $x$. 

Another interesting measure in quantum information is the Gaussian entanglement of formation. This is a measure to quantify the amount of entanglement needed to prepare a state of two systems under so-called LOCC operations (local quantum operations on each part of the systems and classical communication between the parts). It was shown in \cite{wol04} that this measure can be written as
\begin{align*}
	E_{\mathrm{form}}(\gamma_{AB})=\min \{H(\gamma_p)| \gamma_{AB}\geq S^TS, S\in Sp(2n)\}
\end{align*}
where $\gamma_p$ is the reduced state of $S^TS$ and $H(\cdot)$ denotes the entropy of entanglement. Using the methods of \cite{ide16} and the stability results of this paper, one can now prove:
\begin{prop}
The Gaussian entanglement of formation is continuous on the interior of the set of covariance matrices.
\end{prop}
\begin{proof}[Sketch of the proof] This can be proven in two ways. For $\gamma_{AB}$ in the interior of the set of covariance matrices, one can either use set-valued analysis as in the proof of Theorem 4.4 in \cite{ide16} to prove that the set $\{\gamma_{AB}\geq \gamma_0\geq i\sigma\}$ is convex and varies continuously with $\gamma_{AB}$. Then the result follows from Corollary \ref{cor:entropy}. 

Equivalently, one can use Theorem \ref{thm:stabilitysmaller} to show that for any $\varepsilon>0$ and any symmetric $E$ small enough, for any $S^TS\leq \gamma_{AB}$ there exists $S_{\varepsilon}^TS_{\varepsilon}\leq \gamma_{AB}+\varepsilon E$, such that their norm difference is small and then apply Corollary \ref{cor:entropy}. 
\end{proof}

As a last application, let us mention that the stability of $S^TS$ as in Theorem \ref{thm:stabilitysmaller} is implicitly useful in \cite{ide16}: There, we provide a program to compute an operational measure for squeezing. Given a covariance matrix $\gamma$ of a state to be constructed, the program first computes Williamson's normal form and takes $S^TS$ as a starting point. If $S^TS$ was not continuous in the covariance matrix, this would imply that rounding errors in $\gamma$ could result in the corresponding $S^TS$ not being a feasible point for the program. Theorem \ref{thm:stabilitysmaller} asserts that this problem cannot occur.

\paragraph*{Acknowledgements}
M.I. thanks the Studienstiftung des deutschen Volkes for financial support.

\printbibliography

\appendix
\section{Some useful lemmata} \label{app:calc}
In this appendix, we will first review the main nontrivial theorems we use in the main text to establish our results. In addition, we give a number of small lemmata that include calculations that are frequently used in the main text to eliminate $\varepsilon$-dependencies of the constants. Since these are not very important for the gist of the argument, they are collected here in order not to further clutter the main text.

\begin{lem}[\cite{wil65} Chapter 2 Section 10] \label{lem:wilkeigen}
Let $A$ be a Hermitian matrix with nondegenerate spectrum and $B$ be a perturbation with $\|B\|_{\infty}\leq 1$. Then there exists a number $c_{\mathrm{vec}}>0$ such that for all $c_{\mathrm{vec}}>\varepsilon>0$ we have 
\begin{align}
	\|x_i-x_i(\varepsilon)\|_2 \leq\frac{2n}{\min_{i\neq j}|\lambda_i-\lambda_j|}\varepsilon
\end{align}
where $x_i$ denotes the $i$-th eigenvector of $A$ and $x_i(\varepsilon)$ the $i$-th eigenvalue of $A+\varepsilon B$.
\end{lem}
\begin{proof}
Since this is not the exact formulation of the section in \cite{wil65}, a few words on how this theorem is related to what is written there: The section computes the first order term in the perturbative expansion of an eigenvector $x_i(\varepsilon)$. Since Hermitian eigenvectors are orthogonal the first order term of the eigenvector expansion of $x_1$ as in (10.2) of \cite{wil65} reads
\begin{align*}
	\leq \varepsilon \left(\frac{\beta_{21} x_2}{\lambda_1-\lambda_2}+\ldots +\frac{\beta_{n1} x_n}{\lambda_n-\lambda_2}\right),
\end{align*}
where $|\beta_{21}|\leq \|B\|_{\infty}\leq 1$ are some numbers and $\lambda_i$ are the eigenvalues of $A$. Hence for $\varepsilon$ small enough, we have
\begin{align*}
	\|x_i-x_i(\varepsilon)\|_2 &\leq \left(\frac{\|x_2\|_2}{\lambda_1-\lambda_2}+\ldots +\frac{\|x_n\|_2}{\lambda_n-\lambda_2}\right)\varepsilon +\mathcal{O}(\varepsilon^2) \\
	&\leq \frac{n}{\min_{i\neq j}|\lambda_i-\lambda_j|}\varepsilon +\mathcal{O}(\varepsilon^2).
\end{align*}
For $\varepsilon$ small enough, this then implies the bound in the theorem.
\end{proof}

\begin{lem}[\cite{bha96} Theorem VIII.3.9] \label{lem:bhatialemma}
Let $A$, $B$ be any two matrices such that $A=SD_1S^{-1}$, $B=TD_2T^{-1}$, where $S,T$ are invertible matrices and $D_1,D_2$ are real diagonal matrices. Then
\begin{align}
	\|\mathrm{Eig}^{\downarrow}(A)-\mathrm{Eig}^{\downarrow}(B)\|\leq (\kappa(S)\kappa(T))^{1/2}\|A-B\|
\end{align}
for every unitarily invariant norm. Here, $\kappa$ is the condition number and $\mathrm{Eig}^{\downarrow}$ denotes the (ordered) set of eigenvalues. 
\end{lem}

\begin{lem}[\cite{bha96} Theorem VII.3.2] \label{lem:stablesubspace}
Let $A,B\in \C^{n\times n}$ be Hermitian operators and let $S_1, S_2$ be any two subsets of $\R$ such that $\operatorname{dist}(S_1,S_2)=\delta>0$. Let $E=P_A(S_1)$ ($F=P_B(S_2)$) be the spectral projection onto the space spanned by the eigenvectors of $A$ ($B$) corresponding to eigenvalues in $S_1$. Then, for every unitarily invariant norm,
\begin{align}
	\|EF\|\leq \frac{\pi}{2\delta} \|A-B\|
\end{align}
\end{lem}

\begin{lem}\label{lem:squareroot}
Let $A,B\in \C^{n\times n}$ be positive semidefinite operators. Then, for every unitarily invariant norm, 
\begin{align}
	\|A^{1/2}-B^{1/2}\|\leq \|A-B\|^{1/2}_{\infty} \|\id\| \label{eqn:square1}
\end{align}
\end{lem}
\begin{proof}
This follows directly from the proof of Theorem X.1.1 in \cite{bha96} using that the square root function is operator monotone on positive semidefinite matrices and $0^{1/2}=0$. 
\end{proof}

\begin{lem}\label{lem:inverse}
Let $A,B\in \C^{n\times n}$ be positive definite operators. Then for every unitarily invariant norm,
\begin{align}
	\|A^{-1}-B^{-1}\|\leq \|A^{-1}\|\|B^{-1}\|\|A-B\|
\end{align}
\end{lem}
\begin{proof}
Calculate:
\begin{align*}
	\|A^{-1}-B^{-1}\|&=\|A^{-1}(\id-AB^{-1})BB^{-1}\|\leq \|A^{-1}\|\|B^{-1}\|\|(\id-AB^{-1})B\| \\
		&\leq \|A^{-1}\|\|B^{-1}\|\|A-B\|
\end{align*}
\end{proof}

\begin{lem}\label{lem:inversenorm}
Let $M$ be an invertible matrix, $E$ a matrix with $\|E\|_{\infty}=1$ and $\|M^{-1}\|_{\infty}\leq \frac{1}{2\varepsilon}$, then
\begin{align}
	\|(M+\varepsilon E)^{-1}\|_{\infty}\leq 2\|M^{-1}\|_{\infty}
\end{align}
\end{lem}
\begin{proof}
Using the Woodbury formula (which was not found by Woodbury \cite{hag89}), we have:
\begin{align*}
	(M+\varepsilon E)^{-1}=M^{-1}-M^{-1}(I+\varepsilon EM^{-1})^{-1}\varepsilon EM^{-1}.
\end{align*}
Since $\|M^{-1}\|_{\infty}\leq \frac{1}{2\varepsilon}$, the Neuman series of $(I+\varepsilon EM^{-1})^{-1}$ converges and we have $(I+\varepsilon EM^{-1})^{-1}=\sum_{n=0}^{\infty}\varepsilon^n(EM^{-1})^n$, and hence $\|(I+\varepsilon EM^{-1})^{-1}\|_{\infty}\leq \sum_{n=0}^{\infty} \varepsilon^n \|M^{-1}\|_{\infty}^n\leq 2$, which implies:
\begin{align*}
	\|(M+\varepsilon E)^{-1}\|_{\infty} \leq \|M^{-1}\|_{\infty}+\|M^{-1}\|_{\infty}\cdot 2\varepsilon \|EM^{-1}\|_{\infty}.
\end{align*}
Finally, since $\varepsilon \|M^{-1}\|_{\infty}\leq 1/2$ by assumption, we have $\|M^{-1}\|_{\infty}\cdot 2\varepsilon \|EM^{-1}\|_{\infty}\leq \|M^{-1}\|_{\infty}\cdot 2\varepsilon \|M^{-1}\|_{\infty}\leq \|M^{-1}\|_{\infty}$. 
\end{proof}

\begin{lem} \label{lem:kappaepsilon}
Let $M,E\in \R^{n\times n}$, $M\geq 0$ and $\|E\|_{\infty}=1$. If $\|M^{-1}\|_{\infty}\leq \frac{1}{2\varepsilon}$ and $\varepsilon<\|M\|_{\infty}$, we have
\begin{align}
	\kappa(M+\varepsilon E)\leq 4\kappa(M).
\end{align}
\end{lem}
\begin{proof}
We use $\kappa(M+\varepsilon E)=\|M+\varepsilon E\|_{\infty}\|(M+\varepsilon E)^{-1}\|_{\infty}$ by definition and apply Lemma \ref{lem:inversenorm} to obtain:
\begin{align*}
	\kappa(M+\varepsilon E)\leq 2\|M+\varepsilon E\|_{\infty}\|M^{-1}\|_{\infty}
\end{align*}
Using $\|M+\varepsilon E\|_{\infty}\leq \|M\|+\varepsilon\leq 2\|M\|$ finishes the proof.
\end{proof}

\end{document}